\documentclass[a4paper,10pt]{amsart}

\usepackage{amsmath}
\usepackage{amssymb}
\usepackage{amsthm}
\usepackage{upgreek} 
\usepackage{listings}

\usepackage[backgroundcolor=green!15]{todonotes}
\presetkeys{todonotes}{inline}{}

\newtheorem{theorem}{Theorem}[section]
\newtheorem{lemma}[theorem]{Lemma}

\newtheorem{remark}[theorem]{Remark}

\theoremstyle{definition}

\newcommand{\Z}{\mathbb Z}

\begin{document}
\lstset{language=C++,basicstyle=\footnotesize, numbers=left, numberstyle=\tiny}

\title{Arcs in $\Z^2_{2p}$}

\author[Z. St\c{e}pie\'{n}]{Zofia St\c{e}pie\'{n}}
\address{
School of Mathematics, West Pomeranian University of Technology\\
al. Piast\'{o}w 48/49, 70-310 Szczecin, Poland}
\email{stepien@zut.edu.pl}

\author[L. Szymaszkiewicz]{Lucjan Szymaszkiewicz}
\address{
Institute of Mathematics, Szczecin University\\ 
Wielkopolska 15, 70-451 Szczecin, Poland
}
\email{lucjan.szymaszkiewicz@usz.edu.pl}



\begin{abstract}
An arc in $\Z^2_n$ is defined to be a set of points no three of which are collinear.
We describe some properties of arcs and determine the maximum size of arcs for 
some small $n$.
\end{abstract}

\maketitle

\section{Introduction}

The no-three-in-line problem 
is an old question (see \cite{Dud}) which
asks for the maximum number of points that can be placed in the 
$n \times n$ grid with no three points collinear. 
This question has been widely studied (see \cite{E}, \cite{F92}, \cite{F98}, \cite{GK} and \cite{HJSW}),
but is still not resolved.

In this paper we consider the no-three-in-line problem 
over the $\Z_n^2$.
This modified problem is still interesting and was investigated in 
\cite{Huizenga}, \cite{Kurz09} and \cite{T1}.
Many authors considered arcs in the context of projective geometry, see e.g. \cite{Bose}, \cite{Segre}
and \cite{Hirschfeld} for further reference, and in a context of Hjelmslev geometry, see e.g. \cite{Honold01}, \cite{Honold05}, \cite{Kurz11} and \cite{Kleinfeld} for definition of abstract Hjelmslev plane.

Four integers $a,$ $b,$ $u,$ $v$ with $\gcd(u,v)=1$
correspond to the line $\left\{(a+u k, b+v k):k\in\Z  \right\}$ on $\Z^2$.
We define lines in $\Z^2_n$ to be images, by natural projection, of lines in the $\Z^2$ (see \cite{T1}).
We say that points $a_1,\ldots,a_k$ are in line on $\Z^2_n$ (or collinear) if there exists a line $l$
in $\Z^2_n$ such that $a_1,\ldots,a_n\in l$.
We would like to remark that one could also define a line as a coset of a (maximal) cyclic subgroup of $\Z^2_n$
(see \cite{Huizenga}, \cite{Kurz09}). 

If $n=p$ is a prime 
the resulting space is just $\mathrm{AG}(2,p)$ where 
every two distinct lines intersect in either $0$ or $1$ point. 
In contrast, for arbitrary $n$ it may happen that two distinct lines meet each other in more than a single point.
 
We call $X\subset \Z^2_n$ an arc if it does not contain any three collinear points, and we call
$X$ complete if it is maximal with respect to the set-theoretical inclusion.
We denote by $\uptau(\Z^2_n)$ the maximum possible size of an arc in $\Z^2_n$.

Kurz \cite{Kurz09} uses the integer linear programming to determine numbers $\uptau(\Z^2_n)$.
He identifies values of $\uptau(\Z^2_n)$ for $n \leq 21$ and, after transformations, the value of $\uptau(\Z^2_{25})$  (see also \cite{Kurz11} for this case). 
We were interested in finding the lacking values of $\uptau(\Z^2_{22})$ and $\uptau(\Z^2_{24})$.
Our first try was using the mathematical programming solver Gurobi \cite{Gurobi}.
We succeed with computing $\uptau(\Z^2_{24})$ but were not able to determine $\uptau(\Z^2_{22})$ 
in this way. Therefore, we investigated some properties of arcs in $\Z^2_{2p}$ for $p$ prime,
which allowed us to compute $\uptau(\Z^2_{22})$.


We will use the following notation.
Denote by $\pi_n$ the natural projection from $\Z$ to $\Z_n$.
If the context is clear, we omit the index $n$. 
For $u=(u_x,u_y)$ we write $\pi(u)$ for $(\pi(u_x), \pi(u_y))$. 
If $p|n$ then by $\phi_p$, we denote the projection $\phi_p:\Z_n \to \Z_p$ 
defined by  $\phi_p \circ \pi_n = \pi_p$.
Similarly, we write $\phi_p(u)$ for $(\phi_p(u_x), \phi_p(u_y))$.

By automorphism  of $\Z_{n}^{2}$ we mean  mapping from $\Z_{n}^{2}$ to $\Z_{n}^{2}$
which preserves arcs. The group formed by these elements is denoted by $\mathcal G_{n}.$
Let $GL_{2}(\Z_{2p})$ denote the group of invertible $2\times 2$ matrices with coefficients
in $\Z_{2p}$. 
We denote by $f_{A}$ the automorphism of $\Z^{2}_{2p}$ corresponding to the matrix $A$.
Another class of automorphisms are translations. In this paper $t_{v}$
denotes the translation by a vector $v$. 
More explicitly, $t_{(v_x,v_y)}(u_x,u_y)=(u_x+v_x, u_y+v_y)$.

By $D(u,v,w)$ we denote
\[
\begin{vmatrix}
1 & 1 & 1 \\
u_x & v_x & w_x \\
u_y & v_y & w_y \\
\end{vmatrix}
\]
where $u=(u_x, u_y)$, $v=(v_x, v_y)$ and $w=(w_x, w_y)$.

\section{Basic facts}

For a prime $p$ the exact value of $\uptau\left(\Z^{2}_{p}\right)$ is well-known. See, for instance, 
\cite{T1}.
\begin{lemma}\label{L:pp}
Let p be an odd prime, then $\uptau\left(\Z^{2}_{p}\right)=p+1.$
\end{lemma}

The following bound was proven in \cite{Kurz09}.
\begin{lemma}\label{L:Kurz09}
$\uptau\left(\Z^{2}_{mn}\right)\leq \min\left\{m\cdot\uptau\left(\Z^{2}_{n}\right),n\cdot\uptau\left(\Z^{2}_{m}\right)\right\}$
for coprime integers $m, n>1$.
\end{lemma}

Recall the well-known test to check whether three points are collinear or not.
\begin{lemma}\label{T:det-p}
Points $a, b, c \in \Z^2_p$ for $p$ prime are in a line if and only if $D(a,b,c)=0$.
\end{lemma}

The following lemma is adapted from Huizenga \cite{Huizenga}.  
\begin{lemma}\label{L:Huizenga}

Let $N = m\cdot n$ with coprime $m$ and $n$. Then any three points in $\Z^{2}_{N}$ are collinear
if and only if both projections onto $\Z^2_m$ and $\Z^2_n$ give a collinear point set.
\end{lemma}
\begin{theorem}\label{T:det-squarefree}
Let $p_1, p_2$ be primes such that $p_1 \neq p_2$.
Then points $a, b, c \in \Z^2_{p_1 \cdot p_2}$ are in a line if and only if $D(a,b,c)=0$.
\end{theorem}
\begin{proof}
Assume that $D(a,b,c)=0\in\Z_{p_1 \cdot p_2}$. 
Consequently, $D(\phi_{p_i}(a), \phi_{p_i}(b), \phi_{p_i}(c))=0\in\Z_{p_i}$ and 
by Lemma~\ref{T:det-p} $\phi_{p_i}(a), \phi_{p_i}(b), \phi_{p_i}(c)$ are in a line on $\Z^2_{p_i}$ 
for $i=1,2$. By Lemma~\ref{L:Huizenga} $a,b,c$ are in a line on $Z^2_{p_1 \cdot p_2}$.

Conversely, assume that $a,b,c\in\Z^2_{p_1 \cdot p_2}$ are in a line. 
There exists $A,B,C\in\Z^2$ in a line such that $\pi(A)=a$, $\pi(B)=b$ and $\pi(C)=c$. 
Hence, $D(A,B,C)=0\in\Z$ and consequently $D(a,b,c)=0\in\Z_{p_1 \cdot p_2}$.

\end{proof}

\begin{remark}
Generally, zeroing of determinant is necessary but not sufficient for three points to be collinear.
Let $p^2|m$ for some prime $p$. Then for the points $(0,0),(p,0),(0,p)$ we have 
\[
\begin{vmatrix}
1 & 1 & 1 \\
0 & p & 0 \\
0 & 0 & p\\
\end{vmatrix}
= 0
\]
but these point are not collinear.
\end{remark}

\section{Arcs in $\Z^{2}_{2p}$}

\begin{theorem}
We have
$\uptau\left(\Z^{2}_{2p}\right)\leq 2p+2$ and
$\uptau\left(\Z^{2}_{2p}\right)=2p+2$ for $p=3,5.$ 
\end{theorem}
\begin{proof}
It  
is an immediate consequence of Lemma~\ref{L:Kurz09}, Lemma~\ref{L:pp}, Figure~\ref{fig_ex_6} and Figure~\ref{fig_ex_10}.
\end{proof}
\begin{remark}
We conjecture that $p=3$ and $5$ are the only values for which the equality holds.
\end{remark}

Define the maps $\alpha_{2} :\Z_{p}\to\Z_{2p},$ and $\alpha_{p} :\Z_{2}\to\Z_{2p}$ by
\[\alpha_{2} (i+p\Z, j+p\Z) = (2i+2p\Z, 2j+2p\Z),\]
\[\alpha_{p} (i+2\Z, j+2\Z) = (pi+2p\Z, pj+2p\Z).\]

\begin{theorem} 
Let $X$ be a complete arc of $\Z^{2}_{p}$. 
Then $\alpha_2(X)$ is the complete arc of $\Z^{2}_{2p}$. 
\end{theorem}
\begin{proof}
Let $X$ be a complete arc of $\Z^2_p$. Then $\alpha_2(X)$ is obviously the arc of $\Z^2_{2p}$.
Let $c\in \Z^{2}_{2p}$. We will show that there are $a,b\in \alpha_2(X)$ 
such that $a,$ $b,$ $c$ are collinear.

Assume first that $c\in \phi^{-1}_2(0,0).$ Then there is $c'\in Z^{2}_{p}$ such that $\alpha_2(c')=c.$
Since  $X$ is a complete arc then there exist $a', b' \in X$ such that $a',b',c'$ 
are collinear. By Theorem~\ref{T:det-p}, $D(a',b',c')=0\in \Z_{p}.$
Then $D(\alpha_2(a'),\alpha_2(b'),c)=0\in \Z_{2p}.$
 By Theorem~\ref{T:det-squarefree}, $a=\alpha_2(a'),$ $b=\alpha_2(b'),$ $c$ are collinear.

Now assume that $c\notin \phi^{-1}_2(0,0).$ Then there is $v\in \left\{\left[0,p\right],\left[p,0\right],\left[p,p\right]\right\}$ 
such that $t_{v}(c)\in \phi^{-1}_2(0,0).$ Hence the first part of the proof shows that there are $a,b\in \alpha_2(X)$ 
such that $a$, $b$, $t_{v}(c)$ are  collinear (i.e. $D(a,b,t_{v}(c))=0$).
Because $\alpha_{2}(X)\subset \phi^{-1}_{2}(0,0),$
a straightforward calculation shows that $D(a,b,c)=D(a,b,t_{v}(c))$ for 
$a,b\in \alpha_{2}(X).$
Hence $D(a,b,c)=0$ and $a,$ $b,$ $c$ are collinear, by Theorem~\ref{T:det-squarefree}.
This completes the proof. 

\end{proof}

\begin{lemma}
Let $X$ be a complete arc of $\Z^{2}_{2}$. 
Then $\alpha_{p}(X)$ is the complete arc of $\Z^{2}_{2p}$. 
\end{lemma}
\begin{proof}
Let $X$ be a complete arc in 
$\Z^{2}_{2}.$
Then $X=\Z^{2}_{2}$ and $\alpha_p(X)$ is obviously the arc of $\Z^2_{2p}$. 
It is easy to check that for every $c \in Z^{2}_{2p}$ there is $b\in \alpha_{p}(X)\setminus {(0,0)}$ such that $D((0,0),b,c)=0.$
By Theorem~\ref{T:det-squarefree}, $(0,0),$ $b,$ $c$ are collinear.
\end{proof}

The proof of the Theorem \ref{L:init} takes up the rest of this section.
We prepare for the proof by collecting together some useful technical results.

\begin{lemma}\label{L:auto}
Let $\sigma$ be an arbitrary  permutation of the set $\Z^{2}_{2}$. Then there is $f\in \mathcal G_{2p}$ such that 
$\phi_{2}\circ f =\sigma \circ \phi_{2}.$
\end{lemma}

\begin{proof}
Let $f_{A},$ $f_{B}$ be linear transformations determined by matrices $A=\begin{pmatrix}
1 & 0\\
1 & 1\\
\end{pmatrix}$, $B=\begin{pmatrix}
0 & 1\\
1 & 0\\
\end{pmatrix}$, respectively. 
Note that the group of permutations of the set $\Z^{2}_{2}$ is  generated by
transpositions
$\sigma_{1}=\left({\scriptstyle (0,0), (0,1)}\right),$ 
$\sigma_{2}=\left({\scriptstyle (0,1), (1,0)}\right),$ 
$\sigma_{3}=\left({\scriptstyle (1,0), (1,1)}\right).$

One can verify by a straightforward calculation that 
$\phi_{2}\circ f_{A}\circ t_{\left[1,1\right]}=\sigma_{1}\circ \phi_{2},$ 
$\phi_{2}\circ f_{B}=\sigma_{2}\circ \phi_{2} $ and 
$\phi_{2}\circ f_{A}=\sigma_{3}\circ \phi_{2}.$
\end{proof}
\begin{lemma}\label{L:n}
Let $X\subset \Z^{2}_{2p}$ be an arc.
If $\left|\phi_2\left(X\right)\right|\leq 2$, then $|X|\leq p+1.$
\end{lemma}

\begin{proof}
The condition $\left|\phi_2\left(X\right)\right|\leq 2$ means that $\phi_2\left(X\right)$ is collinear in $\Z^{2}_{2}.$ 
Since $X$ is an arc then, by Lemma~\ref{L:Huizenga}, $\phi_p\left(X\right)\subset \Z^{2}_{p}$ is an arc and
$|\phi_p\left(X\right)|=|X|$.
By Lemma~\ref{L:pp}, $|X|\leq p+1.$
\end{proof}

\begin{lemma}\label{L:m1} 
Let $X\subset \Z^{2}_{2p}$ be an arc and $a\in X$.
If $t_{v}(a)\in X$ for some 
$v\in \left\{\left[0,p\right],\left[p,0\right],\left[p,p\right]\right\},$
then $|X|\leq p+3.$
\end{lemma}
\begin{proof} Consider first the case that $v=[0,p].$ Let $l_0=\pi(L_{0})$ and 
$l_1=\pi(L_{1})$ denote the lines in $\Z^{2}_{2}$, 
where $L_0$ and $L_1$ are given by the equations $x=0$, $x=1$, respectively.
 Assume without loss of generality that $a\in \phi_2^{-1}\left(l_{0}\right).$
A straightforward calculation shows that $D(a,t_{[0,p]}(a),b)=0$ for all $b\in \phi_2^{-1}\left(l_{0}\right).$ 
Hence, by Theorem~\ref{T:det-squarefree}, $X\subseteq \left(\phi_2^{-1}\left(l_{1}\right)\cup\left\{a,t_{[0,p]}(a)\right\}\right).$
The same argument as in Lemma~\ref{L:n} shows that 
$\left|X\cap \phi_2^{-1}\left(l_{1}\right)\right| \leq p+1.$ 
The remaining two cases are dealt with similarly.
\end{proof}

\begin{lemma}\label{L:01}
Let $(a_{x},a_{y})\in \Z^{2}_{2p}$ and $a_{y}$  is invertible in $\Z_{2p}.$
If $a_{y}b_{x}^{1}-b_{y}^{1}a_{x}=a_{y}b_{x}^{2}-b_{y}^{1}a_{x}=a_{y}b_{x}^{3}-b_{y}^{1}a_{x}=A$
then $b^{1},$ $b^{2},$ $b^{3}$
 are on line.
\end{lemma}
\begin{proof}
 We have

\[
\begin{aligned}
D\left(b^{1},b^{2},b^{3}\right)&=
a_{y}^{-1}
D\Big(\left(A,b_{y}^{1}\right),
\left(A,b_{y}^{2}\right),
\left(A,b_{y}^{3}\right)\Big)=0.\\
\end{aligned}
\]
The result now follows from Theorem~\ref{T:det-squarefree}.
\end{proof}

\begin{theorem}\label{L:init}
Let $p\geq 5$ and $X\subset \Z^{2}_{2p}$ be an arc. If $|X|>p+3$, then there is $f\in \mathcal G_{2p}$ such that $(0,0),$ $(1,0),$ $(0,1)\in f(X).$
\end{theorem}
\begin{proof}
By Lemma~\ref{L:n} and Lemma~\ref{L:auto}  we may assume that 
$\phi_2^{-1}\left(0,0\right)\cap X \neq \emptyset,$
$\phi_2^{-1}\left(0,1\right)\cap X \neq \emptyset$ and 
$\phi_2^{-1}\left(1,0\right)\cap X\neq \emptyset.$  
Since $|X|>8$ we may assume that $|\phi_2^{-1}\left(1,0\right)\cap X|\geq 3,$ by Lemma~\ref{L:auto}.
We may also assume that $(0,0)\in X$. For if not, then we may translate $X$ by a vector $[-u,-v]$ for some $(u,v)\in\phi_2^{-1}\left(0,0\right)\cap X$.
Let $(a_{x},a_{y}) \in \phi_2^{-1}\left(0,1\right)\cap X.$ Then possibilities for $a_{y}$ are: (i) $a_{y}=p,$ 
(ii) $a_{y}$ is invertible in $\Z_{2p}.$

In the case (i), the assumption that $|X|>p+3$ and 
by Lemma~\ref{L:m1} imply that $a_{x}\neq 0.$ Hence $a_{x}b_{y}-b_{x} a_{y}\neq p$ for all 
$(b_{x},b_{y}) \in \phi_2^{-1}\left(1,0\right)\cap X.$

In the case (ii), the assumption that $|\phi_2^{-1}\left(1,0\right)\cap X|\geq 3$ and Lemma~\ref{L:01} imply that there is $(b_{x},b_{y}) \in \phi_2^{-1}\left(1,0\right)\cap X$ such that $a_{x}b_{y}-b_{x} a_{y}\neq p.$

In both cases, there is $(b_{x},b_{y})\in \phi_2^{-1}\left(1,0\right)\cap X$
such that $a_{x}b_{y}-b_{x} a_{y}$ is invertible in $\Z_{2p}.$ 
Then the map $f_{A}$ with matrix $A=\frac{1}{a_{x}\cdot b_{y}-b_{x}\cdot a_{y}}\begin{pmatrix}
b_{y} & -b_{x}\\
-a_{y} & a_{x} \\
\end{pmatrix}$ maps
\[(a_{x},a_{y})\longrightarrow (1,0),\]
\[(b_{x},b_{y})\longrightarrow (0,1),\]
\[(0,0)\longrightarrow (0,0).\]
\end{proof}

\section{Numerical results}

Computing $\uptau(\Z^2_{24})$ was quite easy. We used the mathematical programming solver 
Gurobi \cite{Gurobi}.
After expanding $12946227$ nodes ($602109089$ simplex iterations) the solver found the solution 
depicted in Figure~\ref{fig_ex_24}.
On the other hand, using the solver to compute $\uptau(\Z^2_{22})$ was not successful.

Consequently, we used a more direct approach.
For finding the value of $\uptau(\Z^2_{2p})$ we implemented backtracking search algorithm 
depicted in Figure~\ref{fig_BS}. 
Figure~\ref{fig_ex_10}, Figure~\ref{fig_ex_14} and Figure~\ref{fig_ex_22} show results of our search.
Note that it is easy to find these complete arcs.
The difficult part is showing that there are no bigger ones.

To limit searching space we made some optimizations.
Thanks to Theorem~\ref{L:init} we start with the set $(0,0),(1,0),(0,1)$ (lines 29--31).
Since the initial set is symmetric, after checking the fourth point $(x,y)$ (line 34) we can 
exclude $(y,x)$ from further search (line 35 and 36).
After choosing a point $(x,y)$ (line 3 and 4) we can also exclude 
points $t_{[p,0]}(a)$, $t_{[0,p]}(a)$ and $t_{[p,p]}(a)$ 
from further search (lines 8--11 and Lemma~\ref{L:m1}).
The program presented here is simplified, the real computations were performed in parallel.

For reference we present in Table~\ref{table_uptau} all known values 
of $\uptau(\Z^2_m)$ for nonprime $m$ (bold numbers are computed by us).
\begin{table}[!htbp]
\begin{center}
\footnotesize
\begin{tabular}{|c|ccccccccccccccc|}
 \hline
 $m$         & 4 & 6 & 8 & 9 & 10 & 12 & 14 & 15 & 16 & 18 & 20 & 21 & 22 & 24 & 25 \\
 \hline
 $\uptau(m)$ & 6 & 8 & 8 & 9 & 12 & 12 & 12 & 15 & 14 & 17 & 18 & 18 & \bfseries{18} & \bfseries{20} &20\\
 \hline
\end{tabular}
\caption{Values for $\uptau(\Z_m)$.}
\label{table_uptau}
\end{center}
\end{table}
Recall that $\uptau(\Z^2_2)=4$ and $\uptau(\Z^2_p)=p+1$ for primes $p>2$.
We also attach a table with some lower and upper bounds for non-prime $ 26\leq n\leq 40$ found by computer search or application of Lemma~\ref{L:Kurz09}.
\begin{table}[!htbp]
\begin{center}
\footnotesize
\begin{tabular}{|c|cccccccccccc|}
 \hline
 $m$         & 26    & 27    & 28    & 30    & 32    & 33    & 34    & 35    & 36    & 38    & 39    & 40  \\
 \hline
 $\uptau(m)$ & 20-28 & 20-28 & 22-32 & 23-36 & 23-35 & 23-36 & 22-36 & 26-40 & 25-36 & 23-40 & 24-42 & 26-40\\
 \hline
\end{tabular}
\caption{Bounds for $\uptau(\Z_m)$.}
\label{table_uptau2}
\end{center}
\end{table}
\subsection*{Acknowledgment}
The authors are grateful to the reviewer whose valuable suggestions resulted in a better
organized and improved paper.


\begin{figure}
\begin{lstlisting}
// GLOBALS: N - size of the torus, best_solution
void select_point(int x, int y, Torus& torus, Solution& solution) {
    torus.set_IN(x,y);
    solution.add_point(x,y);
    if(solution.size() > best_solution.size()) {
        best_solution = solution;
    }
    int x2 = (x+N/2)%N, y2 = (y+N/2)%N;
    torus.set_OUT(x,y2);
    torus.set_OUT(x2,y);
    torus.set_OUT(x2,y2);
    for(int i=0; i<solution.size()-1; ++i) {
        torus.mark_all_lines_through_two_points(x, y, 
                   solution.get(i).getx(), solution.get(i).gety());
    }			
}
void attempt(int x, int y, Torus torus, Solution solution) {
    select_point(x,y,torus,solution);
    if(solution.size()+torus.number_of_FREE()<=best_solution.size()) {
        return; // backtrack
    }
    while(torus.find_next_FREE(x,y)) {
        attempt(x,y,torus,solution);
        torus.set_OUT(x,y);
    }
}
int main() {
    Torus torus; Solution solution;
    select_point(0,0,torus,solution);
    select_point(1,0,torus,solution);
    select_point(0,1,torus,solution);
    int x = 1, y = 1;
    while(torus.find_next_FREE(x,y)) {
        attempt(x,y,torus,solution);
        torus.set_OUT(x,y); 
        torus.set_OUT(y,x);
    }
}
\end{lstlisting}
\caption{Backtracking search}
\label{fig_BS}
\end{figure}


\begin{figure}
  \begin{center}
    \setlength{\unitlength}{0.15cm}
    \begin{picture}(6,6)
      \multiput(-0.5,-0.5)(0,1){7}{\line(1,0){6}}
      \multiput(-0.5,-0.5)(1,0){7}{\line(0,1){6}}
      \put(0,0){\circle*{0.75}}
      \put(1,0){\circle*{0.75}}
      \put(0,1){\circle*{0.75}}
      \put(2,1){\circle*{0.75}}
      \put(1,2){\circle*{0.75}}
      \put(2,2){\circle*{0.75}}
      \put(5,3){\circle*{0.75}}
      \put(3,5){\circle*{0.75}}
    \end{picture}
  \end{center}
  \caption{A complete arc set of cardinality $8$ over $\Z^2_6$.}
  \label{fig_ex_6}
\end{figure}

\begin{figure}
  \begin{center}
    \setlength{\unitlength}{0.15cm}
    \begin{picture}(10,10)
      \multiput(-0.5,-0.5)(0,1){11}{\line(1,0){10}}
      \multiput(-0.5,-0.5)(1,0){11}{\line(0,1){10}}
      \put(0,0){\circle*{0.75}}
      \put(1,0){\circle*{0.75}}
      \put(0,1){\circle*{0.75}}
      \put(1,1){\circle*{0.75}}
      \put(4,2){\circle*{0.75}}
      \put(7,2){\circle*{0.75}}
      \put(5,4){\circle*{0.75}}
      \put(6,4){\circle*{0.75}}
      \put(5,7){\circle*{0.75}}
      \put(6,7){\circle*{0.75}}
      \put(4,9){\circle*{0.75}}
      \put(7,9){\circle*{0.75}}
    \end{picture}
  \end{center}
  \caption{A complete arc set of cardinality $12$ over $\Z^2_{10}$.}
  \label{fig_ex_10}
\end{figure}

\begin{figure}
  \begin{center}
    \setlength{\unitlength}{0.15cm}
    \begin{picture}(14,14)
      \multiput(-0.5,-0.5)(0,1){15}{\line(1,0){14}}
      \multiput(-0.5,-0.5)(1,0){15}{\line(0,1){14}}
      \put(0,0){\circle*{0.75}}
      \put(1,0){\circle*{0.75}}
      \put(0,1){\circle*{0.75}}
      \put(1,1){\circle*{0.75}}
      \put(3,2){\circle*{0.75}}
      \put(4,2){\circle*{0.75}}
      \put(9,3){\circle*{0.75}}
      \put(10,3){\circle*{0.75}}
      \put(3,7){\circle*{0.75}}
      \put(12,7){\circle*{0.75}}
      \put(4,8){\circle*{0.75}}
      \put(7,10){\circle*{0.75}}
    \end{picture}
  \end{center}
  \caption{A complete arc set of cardinality $12$ over $\Z^2_{14}$.}
  \label{fig_ex_14}
\end{figure}

\begin{figure}
  \begin{center}
    \setlength{\unitlength}{0.15cm}
    \begin{picture}(22,22)
      \multiput(-0.5,-0.5)(0,1){23}{\line(1,0){22}}
      \multiput(-0.5,-0.5)(1,0){23}{\line(0,1){22}}
      \put(0,0){\circle*{0.75}}
      \put(1,0){\circle*{0.75}}
      \put(0,1){\circle*{0.75}}
      \put(1,1){\circle*{0.75}}
      \put(3,2){\circle*{0.75}}
      \put(6,2){\circle*{0.75}}
      \put(2,7){\circle*{0.75}}
      \put(20,7){\circle*{0.75}}
      \put(12,10){\circle*{0.75}}
      \put(21,10){\circle*{0.75}}
      \put(2,13){\circle*{0.75}}
      \put(11,13){\circle*{0.75}}
      \put(3,16){\circle*{0.75}}
      \put(21,16){\circle*{0.75}}
      \put(19,17){\circle*{0.75}}
      \put(6,20){\circle*{0.75}}
      \put(15,21){\circle*{0.75}}
      \put(20,21){\circle*{0.75}}
    \end{picture}
  \end{center}
  \caption{A complete arc of cardinality $18$ over $\Z^2_{22}$.}
  \label{fig_ex_22}
\end{figure}
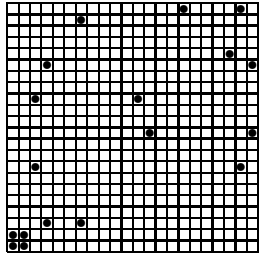

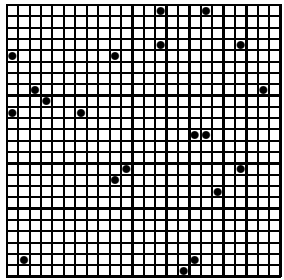
\begin{figure}
  \begin{center}
    \setlength{\unitlength}{0.15cm}
    \begin{picture}(24,24)
      \multiput(-0.5,-0.5)(0,1){25}{\line(1,0){24}}
      \multiput(-0.5,-0.5)(1,0){25}{\line(0,1){24}}
\put(0,14){\circle*{0.75}}
\put(0,19){\circle*{0.75}}
\put(2,16){\circle*{0.75}}
\put(3,15){\circle*{0.75}}
\put(6,14){\circle*{0.75}}
\put(9,8){\circle*{0.75}}
\put(9,19){\circle*{0.75}}
\put(10,9){\circle*{0.75}}
\put(13,20){\circle*{0.75}}
\put(13,23){\circle*{0.75}}
\put(1,1){\circle*{0.75}}
\put(15,0){\circle*{0.75}}
\put(16,1){\circle*{0.75}}
\put(16,12){\circle*{0.75}}
\put(17,12){\circle*{0.75}}
\put(17,23){\circle*{0.75}}
\put(18,7){\circle*{0.75}}
\put(20,9){\circle*{0.75}}
\put(20,20){\circle*{0.75}}
\put(22,16){\circle*{0.75}}
    \end{picture}
  \end{center}
  \caption{A complete arc of cardinality $20$ over $\Z^2_{24}$.}
  \label{fig_ex_24}
\end{figure}


\end{document}